\documentclass{amsart}
\pdfoutput=1
\usepackage{amssymb}
\usepackage{amsmath}
\usepackage{amsthm}
\usepackage{amsbsy}
\usepackage{bm}
\usepackage{graphicx}
\usepackage{cite}
\theoremstyle{plain}
\newtheorem{theorem}{Theorem}[section]
\newtheorem{lemma}[theorem]{Lemma}
\newtheorem{corollary}[theorem]{Corollary}

\theoremstyle{definition}
\newtheorem{definition}[theorem]{Definition}

\theoremstyle{remark}
\newtheorem{remark}[theorem]{Remark}

\newcommand{\R}{\mathbb{R}}

\newcommand{\del}{\partial}

\newcommand{\bs}{\bar\sigma}
\newcommand{\lks}{\mathcal{L}}
\newcommand{\tlks}{\widetilde{\mathcal{L}}}
\newcommand{\bds}{\mathcal{B}}

\begin{document} 

\title{Knots, Braids and First Order Logic}

\author{Siddhartha Gadgil}

\address{	Department of Mathematics,\\
		Indian Institute of Science,\\
		Bangalore 560012, India}

\email{gadgil@math.iisc.ernet.in}

\author{Prathamesh, T\,. V\,. H\,.}


\date{\today}

\begin{abstract}
Determining when two \emph{knots} are equivalent (more precisely \emph{isotopic}) is a fundamental problem in topology. Here we formulate this problem in terms of \emph{Predicate Calculus}, using the formulation of knots in terms of braids and some basic topological results. 

Concretely, Knot theory is formulated in terms of a language with signature $(\cdot,T,\equiv, 1,\sigma,\bar\sigma)$, with $\cdot$ a $2$-function, $T$ a $1$-function, $\equiv$ a $2$-predicate and $1$, $\sigma$ and $\bar\sigma$ constants. We describe a \emph{finite} set of axioms making the language into a (first order) theory. We show that every knot can be represented by a term $b$ in $1$, $\sigma$, $\bs$  and $T$, and knots represented by terms $b_1$ and $b_2$ are equivalent if and only if $b_1\equiv b_2$.

Our formulation gives a rich class of problems in First Order Logic that are important in Mathematics.

\end{abstract}

\maketitle

\section{Introduction}

A (tame) \emph{knot} is a \emph{smooth embedding} of the circle $S^1$ into $3$-dimensional Euclidean space $\R^3$, or equivalently the $3$-sphere $S^3$ (which is viewed as $\R^3$ with an additional point at infinity). We say that two knots are \emph{isotopic} if one can be deformed into the other through smooth embeddings (we give a more precise definition in Section 2. Determining when two knots are equivalent is a fundamental problem in topology.

The goal of this paper is to translate this topological problem into a problem in predicate calculus. We shall in fact formulate the problem of \emph{stable equivalence of links}, which generalises knot equivalence, in terms of predicate calculus. Our formulation is based on the representation of knots in terms of braids.

We give the definitions on knots, links and stable equivalence in Section~\ref{defns}. We then state the axiom system for stable equivalence of links in Section ~\ref{axioms}. We then recall the formulation of knot theory in algebraic terms via Braids in Section~\ref{brdsKnts}. We reformulate this to give a concise description of stable equivalence of links in Section ~\ref{infbrds}. This allows us to prove that our axiom system describes knot theory in Section ~\ref{kntsmodel}. Finally, we give the topological background concerning braids and knots in an Appendix (the paper can be read without this).
 
\section{Knots, Links and stable equivalence}\label{defns}

We begin by recalling some basic definitions. We shall assume that all knots and links are smooth to exclude \emph{wild knots}.

\begin{definition}
 A knot K is defined as the image of a smooth, injective map $h:S^1\rightarrow S^3$ so that $h'(\theta)\neq 0$ for all $\theta\in S^1$.
\end{definition}

\begin{definition}
A link $L\subset S^3$ is a smooth 1-dimensional submanifold of $S^3$ such that each component of $L$ is a knot and there are only finitely many components.
\end{definition}

We shall regard two links as the same if there is an \emph{ambient isotopy} between them, which is defined as follows.

\begin{definition}[Ambient Isotopy] Two links $L_1$ and $L_2$ in $S^3$ are said to be ambient isotopic if there exists a smooth map $F:S^3\times [0,1]\rightarrow S^3$ such that
\begin{enumerate}
 \item $F|_{{S^3}\times \{0\}}$ = $id|_{S^3} : S^3\rightarrow S^3$.
\item $F|_{S^3\times \{1\}}(L_1) = L_2$.
\item $F|_{S^3\times \{1\}}$ is a diffeomorphism $\forall s\in [0,1]$.
\item F is smooth. 
\end{enumerate}
\end{definition}

This gives an equivalence relation by the following well-known theorem.
\begin{theorem}
 Ambient Isotopy induces an equivalent relation on the set of all links.
\end{theorem}

To give a description of knot theory in terms of predicate calculus, we introduce another equivalence relation on links which we call \emph{stable equivalence}.

\begin{definition}
A link $L'$ is said to be a stabilisation of a link $L$ if the following conditions hold.
\begin{enumerate}
\item $L'=L\cup L''$ with $L''$ disjoint from $L$.
\item There is a collection $\{D_1,D_2,\dots,D_n\}$ of disjoint, smoothly embedded discs  in $S^3\setminus L$ with $L''=\bigcup\limits_{i=1}^n \del D_i$. 
\end{enumerate}
\end{definition} 

\begin{definition}
Two links $L_1$ and $L_2$ are said to be stably equivalent, denoted $L_1\equiv L_2$, if there are stabilisations $L'_1$ and $L'_2$ of $L_1$ and $L_2$, respectively, that are ambient isotopic.  
\end{definition}

It is easy to see that $\equiv$ is an equivalence relation. The following result follows from the prime decomposition theorem of Knesser-Milnor (see, for example, \cite{kauffman1987knots}).

\begin{theorem}
If $K_1$ and $K_2$ are knots (regarded as links), then $K_1\equiv K_2$ if and only if $K_1$ is ambient isotopic to $K_2$.
\end{theorem}

We shall denote by $\lks$ the set of equivalence classes of links up to stable equivalence and by $\tlks$ the set of ambient isotopy classes of links. Thus $\lks$ is a quotient of $\tlks$.

\section{Axioms in First Order Logic}\label{axioms}

In this section, we list a set of axioms which enable us to describe stable equivalence of links in terms of first order logic with equality. In later sections we will further substantiate on why these axioms  suffice. Consider a language with signature $(\cdot, T, \equiv, 1, \sigma , \bs)$ such that $\cdot$ is a $2$-function , T is a $1$-function, $\equiv$ is a $2$-predicate, while $1$, $\sigma$ and $\bs$ are constants. The system of axioms for the infinite braid group in this language are the following.

\begin{itemize}

\item Group Axioms (for the set of closed terms)
\begin{enumerate}
\item $\forall x,y,z \quad (x \cdot (y \cdot z) = ((x\cdot y) \cdot z)$

\item$\forall x \quad 1 \cdot x
 = x$
\item $\forall x \quad  x \cdot 1 = x$
\item$ \sigma \cdot \bs = \bs \cdot \sigma = 1$ 
\end{enumerate}
\item Shift operation
\begin{enumerate}
\item $ \forall x,y \quad          T(x \cdot y) = T(x)\cdot T(y)$
\item $  T(e) = e $
\end{enumerate}
\item Braid axioms 
\begin{enumerate}
\item $ \sigma \cdot T(\sigma) \cdot \sigma = T(\sigma) \cdot \sigma \cdot T(\sigma)$
\item $  \forall b\quad  \sigma \cdot T^2(b) = T^2(b) \cdot \sigma $
\end{enumerate}
\item  Equivalence relation
\begin{enumerate}
\item $ \forall x  \quad x\equiv x$
\item $  \forall x,y  \quad x\equiv y\implies y\equiv x$
\item $  \forall x,y,z  \quad x\equiv y\wedge y\equiv z\implies x\equiv z$
\end{enumerate}
\item Markov moves
\begin{enumerate}
\item $ \forall x,y,z \quad y\cdot z=1\implies x\equiv y\cdot x\cdot z$
\item $  \forall x \quad x\equiv \sigma\cdot T(x)$
\item $  \forall x \quad x\equiv \bs\cdot T(x)$
\end{enumerate}
\end{itemize}
Any model of these axioms will be referred to as a \textit{ link model} and the axioms will be referred to as \textit{link axioms}.
\section{Algebraic Formulation of knot theory}\label{brdsKnts}

In this section, we recall the algebraic formulation of knots in terms of Braids. We first recall the definition of braid groups.

\begin{definition}
The braid group $B_n$ is the group generated by $\sigma_1,\sigma_2,\dots,\sigma_{n-1}$ with the relations
\begin{enumerate}
\item $ \sigma_i \cdot  \sigma_j = \sigma_j \cdot  \sigma_i $, where $1 \leq i,j \leq n-1$, $i\geq j+2$.
\item $\sigma_i \cdot \sigma_{i+1} \cdot \sigma_i = \sigma_{i+1} \cdot  \sigma_{i}\cdot \sigma_{i+1}$, where $1\leq i \leq n-2$.
\end{enumerate}
\end{definition}

Note that for $m<n$, there is a natural inclusion homomorphism $i_m$ from $B_m$ to $B_n$ mapping a generator $\sigma_i$ of $B_m$ to the generator $\sigma_i$ of $B_n$. We can thus identify elements in $B_m$ with elements in $B_n$. From now we shall  refer to elements of $\bigcup_{n\in\mathbb{N}\setminus {\{1\}}} B_n$ as braids.

Given an integer $m>1$, we associate a diagram to every generator of the braid group $B_m$, as in the following diagram.

\begin{center}
\includegraphics[scale=0.25]{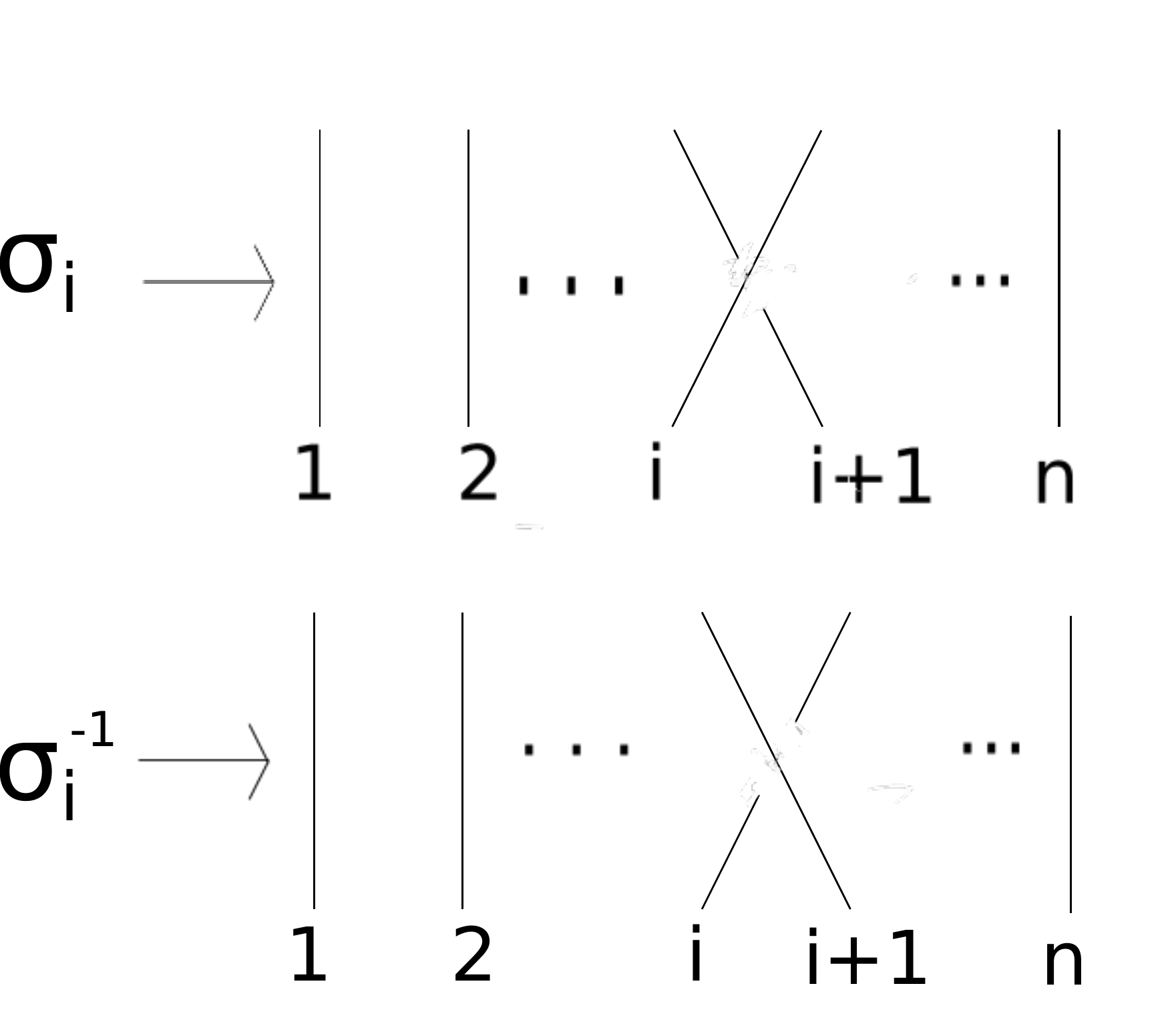}
\end{center}

\newpage
One can thus associate a diagram to every element of $B_m$ by defining the diagram associated to the product of two elements as in the figure below

\begin{center}
\includegraphics[scale=0.40]{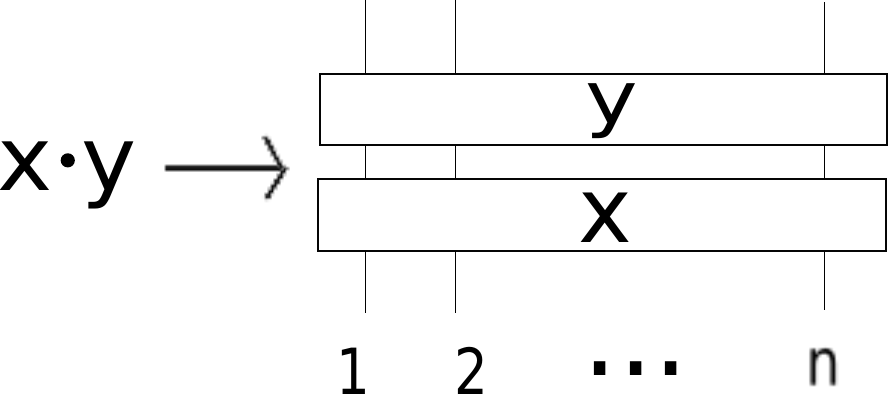}
\end{center}

Given an integer $m>1$ and a braid $b\in B_m$, we can associate to the braid a  link $\lambda(b,m)$ by closing up the diagram associated to a braid and smoothening the sharp edges. A more rigorous construction is provided in the Appendix ~\ref{append}

\begin{center}
\includegraphics[scale=0.25]{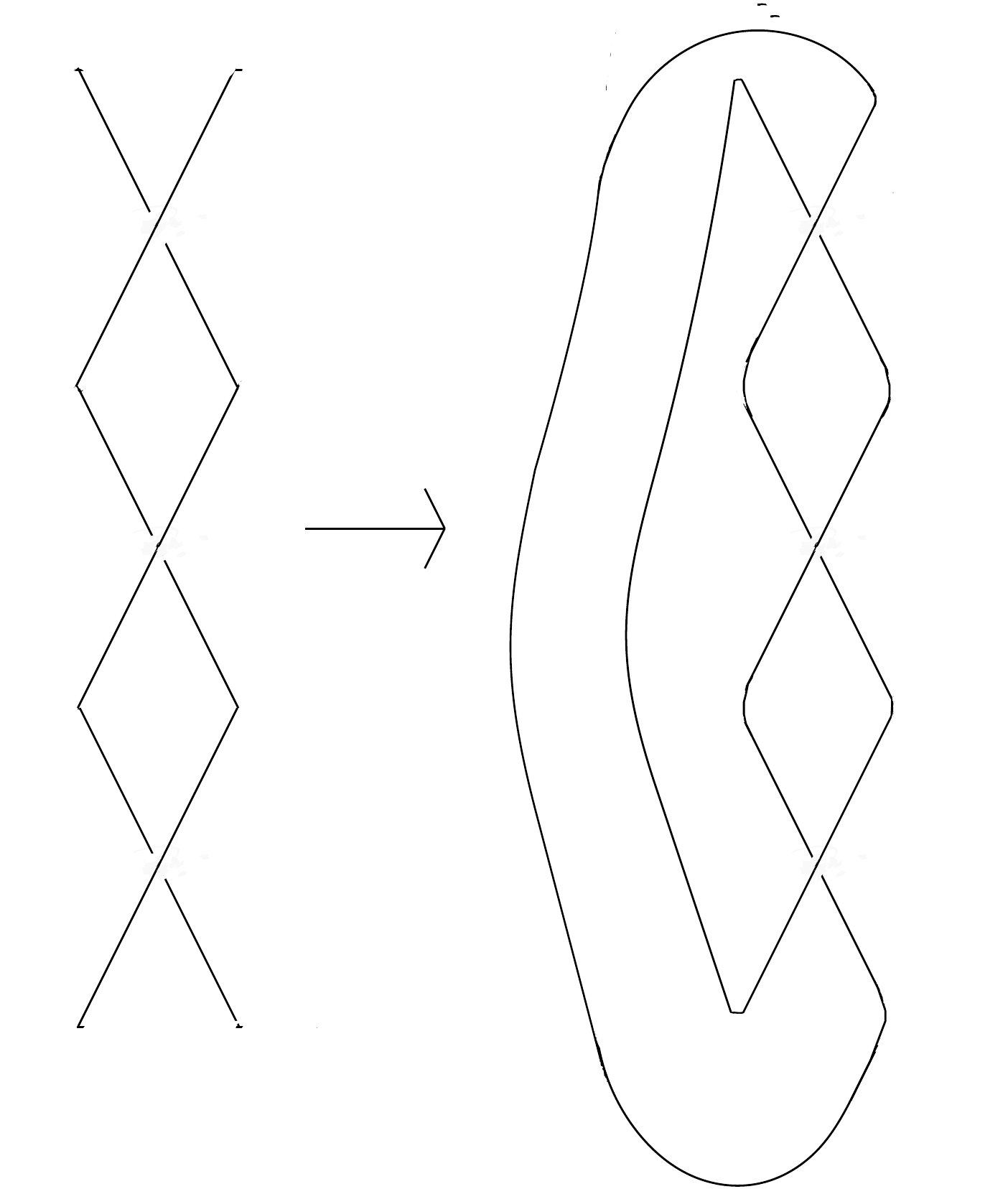}
\end{center}

This gives a function $\lambda:\bds\to\tlks$ from the set
$$\bds=\{(b,m) : b\in B_m\}$$
to the set of links. The following result says that all links are obtained by this construction.
\begin{theorem}[Alexander]\label{Alexander}
For every link $L$, there is an integer $m>1$ and a braid $B\in B_m$ so that $L$ is ambient isotopic to $\lambda(b,m)$.
\end{theorem}

\begin{remark}\label{remarkov}
For $n>m$ and $b\in B_m\subset B_n$, $\lambda(b,m)$ is not ambient isotopic to $\lambda(b,n)$. However it is easy to see that the links are stably equivalent.
\end{remark}

The following lemma is immediate from the fact that a braid $\prod_{k=1}^m \sigma_{i_k}^{\epsilon_k} \in B_n$, where $\epsilon_k$ is $1$ or $-1$ and $i_k \in \mathbb{N}$, when viewed from the other side of the plane in which the braid lies is represented by $\prod_{k=1}^m \sigma_{n-i_k}^{\epsilon_k}$.

\begin{center}
\includegraphics[scale=0.3]{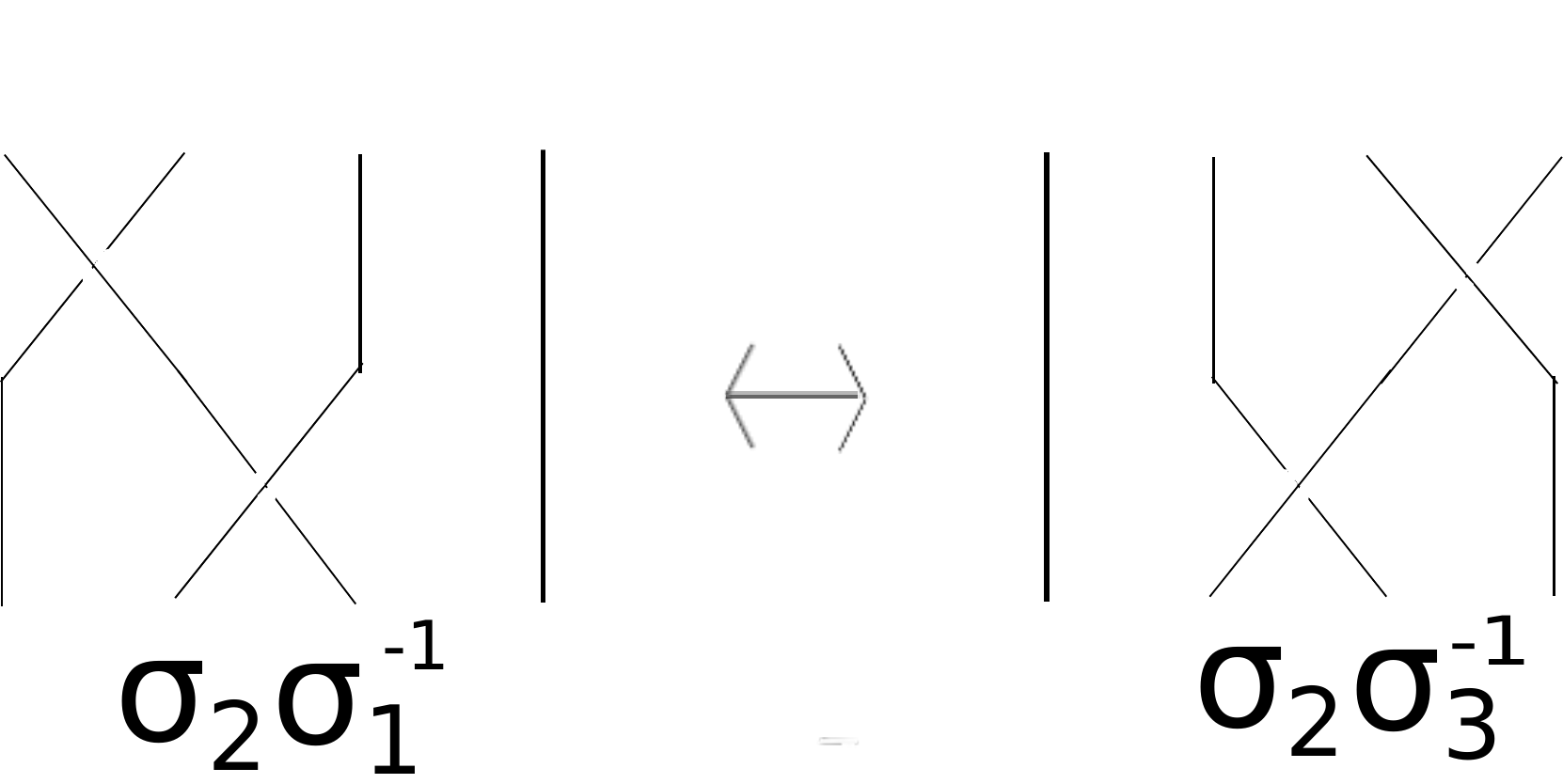}
\end{center}

\begin{lemma} The braid elements  $\prod_{k=1}^m \sigma_{i_k}^{\epsilon_k} $ and $\prod_{k=1}^m \sigma_{n-i_k}^{\epsilon_k}$ are associated to the same link.
 \end{lemma}

To formulate knot theory in terms of braids, we also need to know when two braids correspond to the same link. To state the result of Markov giving such a characterisation, consider the equivalence relation $\sim$ on the set $\bds$ 
generated by
\begin{itemize}
\item For $a,b\in B_m$, $(b,m)\sim (aba^{-1},m)$.
\item For $b\in B_m$, $(b,m)\sim (\sigma_mb, m+1)$.
\item For $b\in B_m$, $(b,m)\sim (\sigma^{-1}_mb, m+1)$.
\end{itemize}

The relation $\sim$ on the set $\bds$ is called the \textit{Markov equivalence}. The first move corresponds to inserting a braid and its inverse below and above an existing braid respectively. The second corresponds to adding a strand to the right of existing braid in such a way that it crosses the strand previously at the extreme right, below the braid. When viewed from the other side of the plane of the braid it corresponds to shifting the braid to the right by a position, inserting a strand in the first position in such a way that it crosses the strand in second position below the braid. 

\begin{center}
\includegraphics[scale=0.3]{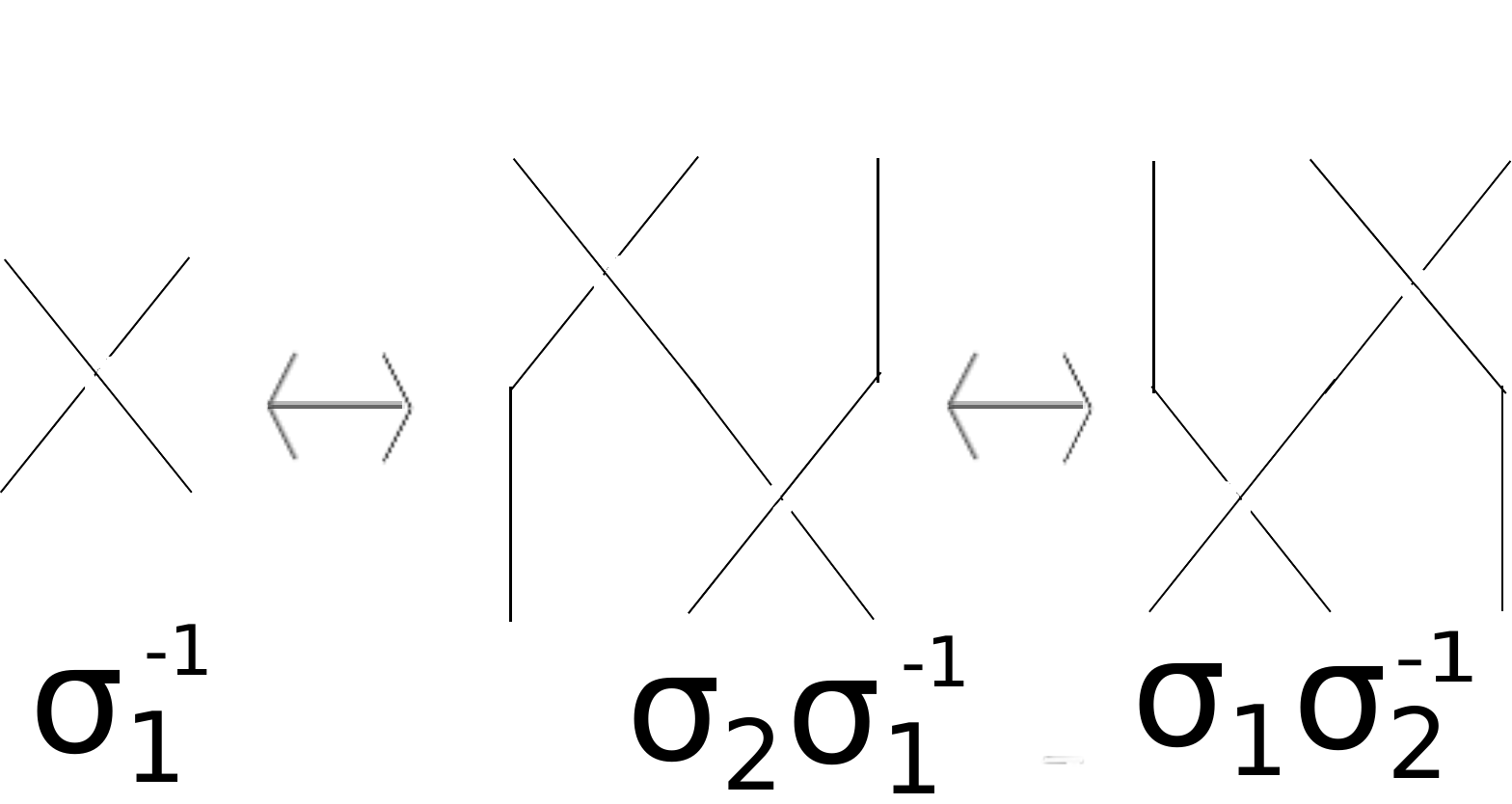}
\end{center}

Thus the second and third Markov moves, could be rewritten to give us the following theorem.

\begin{theorem}\label{altmarkov} The equivalence relation $\cong$ generated by the relations
\begin{enumerate}
\item For $a,b\in B_m$, $m>1$, $(b,m)\cong (aba^{-1},m)$.
\item For $i_k\leq m-1$, $(\prod_{k=1}^m \sigma_{i_k}^{\epsilon_k} ,m)\cong (\sigma_1\prod_{k=1}^m \sigma_{{i_k}+1}^{\epsilon_k}, m+1)$.
\item For $i_k\leq m-1$, $(\prod_{k=1}^m \sigma_{i_k}^{\epsilon_k} ,m)\cong(\sigma^{-1}_1\prod_{k=1}^m \sigma_{{i_k}+1}^{\epsilon_k}, m+1)$.
\end{enumerate}
is the Markov equivalence
\end {theorem}

\begin{theorem}[Markov]\label{markov}
For $i=1,2$, let $m_i>1$ be integers and $b_i\in B_{m_i}$. Then the links $\lambda(b_1,m_1)$ and $\lambda(b_2, m_2)$ are isotopic if and only if $(b_1,m_1)\sim (b_2,m_2)$.
\end{theorem}

To state and equivalent condition for stable equivalence of links, consider the equivalence relation $\approx$ on $\bds$ generated by
\begin{itemize}
\item For any $\beta_1,\beta_2 \in \bds$ such that $\beta_1\sim \beta_2$ , $ \beta_1\approx \beta_2$. 
\item For $m_1,m_2 \in \mathbb{N}$ such that $b \in B_{m_1}, B_{m_2}$, $(b, m_1) \approx (b, m_2)$ 
\end{itemize}

\begin{lemma}\label{stable} Two links are stably equivalent if and only if given $\lambda(b_1,m_1) = l_1$ and $\lambda(b_2, m_2) = l_2$, then  $(b_1,m_1)\approx (b_2,m_2)$.
\end{lemma}
\begin{proof}
The fact that $(b_1,m_1)\approx (b_2,m_2)$ implies $\lambda(b_1,m_1) \equiv \lambda(b_2, m_2)$ follows from the Markov's Theorem and Remark ~\ref{remarkov}. For the converse, let $l_1$ and $l_2$ be two links stably equivalent to each other, such that $\lambda(b_1,m_1) = l_1$ and $\lambda(b_2, m_2) = l_2$. It is easy to see that, for some $k_1,k_2\geq 0$,  $\lambda(b_1, m_1+k_1)$ is isotopic to $\lambda(b_2,m_2+k_2)$ (we can in fact take one of $k_1$ and $k_2$ to be $0$). Then by Markov's theorem, $(b_1,m_1+k)\sim (b_2, m_2)$, hence $(b_1,m_1)\approx (b_2,m_2)$. 
\end{proof}
 
\section{Stable links and Infinite braids}\label{infbrds}

In the previous section, we recalled the well known formulation of knot theory in terms of braid groups. However, to formulate in terms of predicate calculus, we shall reformulate this in terms of a single braid group $B_\infty$. We shall do this by replacing the usual formulation of Markov's theorem by its mirror, which is Theorem ~\ref{altmarkov} above. 

\begin{definition}
The braid group $B_\infty$ is the group generated by the set $\{\sigma_i\}_{i\in \mathbb{N}}$ with the relations
\begin{enumerate}
\item $ \sigma_i \cdot  \sigma_j = \sigma_j \cdot  \sigma_i$, where $i,j\in \mathbb{N}, i\geq j+2$
\item $\sigma_i \cdot \sigma_{i+1} \cdot \sigma_i = \sigma_i \cdot  \sigma_{i+1}\cdot \sigma_i$, where $i \in \mathbb{N }$.
\end{enumerate}
\end{definition}

Note that each braid group $B_n$ can be regarded as a subset of $B_\infty$. Observe that there is an injective homomorphism $T:B_\infty\to B_\infty$, which we call the \emph{shift}, determined by
$$T(\sigma_i) =\sigma_{i+1}.$$

Let $\sigma=\sigma_1$ and $\bs=\sigma_1^{-1}$. Then observe that $\sigma_i=T^{i-1}(\sigma)$. Thus, the translates of $\sigma$  by the semi-group generated by $T$ generate $B_\infty$.

Consider the equivalence relation $\equiv$ on the group $B_\infty$ generated by
\begin{enumerate}
\item For $a,b\in B_\infty$, $aba^{-1}\equiv b$ .
\item For $b\in B_\infty$, $b\equiv \sigma T(b)$ .
\item For $b\in B_\infty$, $b\equiv \bs T(b)$.
\end{enumerate}

We formulate stable equivalence of links in terms of $B_\infty$. 

\begin{theorem}\label{infmarkov}
There is a surjective function $\Lambda: B_\infty\to\lks$ so that, for braids $b_1,b_2\in B_\infty$, $\Lambda(b_1) = \Lambda(b_2)$ if and only if $b_1\equiv b_2$.
\end{theorem}
\begin{proof}
We begin by constructing the function $\Lambda$. For each $n>1$, there is a homomorphism $f_n:B_n\to B_{\infty}$ determined by $f_n(\sigma_i)=\sigma_{i}$. The function is well defined because the braid relations are preserved in $B_{\infty}$. These functions determine a function
$$f:\bds\to B_{\infty} $$
given by
$$f(b, m) = f_m(b)$$
Observe that
$$ f^{-1}(b)= \{(b, m)| b\in\ B_m\}$$
By lemma ~\ref{stable} , the links $\lambda(b, m_1)$ and $\lambda(b, m_2)$ are stably equivalent for $m_1$ and $m_2$ such that $b\in B_{m_1}$ and $b\in B_{m_2}$. Thus $\lambda$ and $f$ induce a function
 $$\Lambda: B_{\infty}\to \lks$$
 given by
$$\Lambda(b) = \lambda(f^{-1}(b))$$

The fact that the map is surjective follows from Alexander's Theorem. Now we prove that $\Lambda(b_1) = \Lambda(b_2) \Leftrightarrow b_1 \equiv b_2$. Consider $b_1, b_2 \in B_{\infty}$ such that $\Lambda(b_1) = \Lambda(b_2)$. Lemma ~\ref{stable} implies that for any $x\in f^{-1}(b_1)$ and $y\in f^{-1}(b_2)$, $x\approx y$. In order to prove that  $\Lambda(b_1) = \Lambda(b_2) \implies b_1 \equiv b_2$, it suffices to prove that for $x, y \in \bds$, $x\approx y \implies f(x)\equiv f(y)$. 
In terms of the generating set of the relation ($\approx$), it translates to proving the following
\begin{enumerate}
\item For $b\in B_{m_1},B_{m_2}$, $f(b, m_1) \equiv f(b, m_2)$ 
\item For $a,b\in B_{m}$, $f(b, m) \equiv f(aba^{-1}, m) $.
\item For $b\in B_{m}$, $f(b, m) \equiv f(\sigma_1 b, m+1) $
\item For $b\in B_{m}$, $f(b, m) \equiv f(\sigma_1^{-1} b, m+1) $.
\end{enumerate}

The relation $f(b, m_1) \equiv f(b, m_2)$  follows from the fact that $f(b_1, m_1) = f(b_1, m_2)$ and $\equiv$ is an equivalence relation.  $f(b, m) \equiv f(aba^{-1}, m) $ follows from the first axiom of generating set of $\equiv$.  The third and fourth conditions are easy to verify.

In order to prove that $b_1\equiv b_2 \implies \Lambda(b_1) = \Lambda(b_2)$, from Theorem 4.5 it suffices to check that for some $a_1\in f^{-1}(b_1)$ and $a_2\in f^{-1}(b_2)$, $a_1 \approx a_2$. Since Markov moves generate the equivalence relation $\equiv$ on $B_{\infty}$ , it suffices to check that if two elements are related to each other by  a Markov move then the elements in their inverse images are $\approx$ equivalent to each other. 

Consider the first Markov move for $B_{\infty}$, $b\equiv aba^{-1}$. For an appropriate $m$ such that $b, aba^{-1} \in B_m$, $(aba^{-1}, m) \sim (b, m)$. Thus from Theorem 4.5, $\lambda(aba^{-1}) = \lambda (b)$. Similarly, $(b, m) \sim (\sigma_1^{\pm 1}b, m+1)$. Thus $b_1\equiv b_2 \implies \Lambda(b_1) = \Lambda(b_2)$.
\end{proof}

\section{Knots as a canonical model}\label{kntsmodel}

From the above we know that equivalence between links can be established by checking if the corresponding elements in the braid group are related by a sequence of moves on $B_{\infty}$ induced by the equivalence relation $\equiv$. Now we prove that  the braid group is the canonical model for the link axioms and thus any model for these axioms can be used to distinguish knots.

\begin{theorem}\label{binfinitylink} $(B_\infty, T, \cdot ,\equiv, \sigma_1, \sigma_1^{-1})$ is a link model.

\end{theorem}
\begin{proof}
We begin by proving that $(B_\infty, T, \cdot ,\equiv, \sigma_1, \sigma_1^{-1})$ satisfies the second braid axiom. From the definitions of $B_{\infty}$ and $T$, it is easy to derive that $(B_\infty, T, \cdot ,\equiv, \sigma_1, \sigma_1^{-1})$ satisfies the other link axioms. Consider an arbitrarily chosen element $b\in B_{\infty}$. Since $B_{\infty}$ is generated by the set $\{\sigma_i\}_{i\in \mathbb{N}}$, $b$ can be represented in terms of generators as  $\prod_{k=1}^n{\sigma_{i_k}^{\epsilon_k}}$, where $i_k\in \mathbb{N}$ and $\epsilon_k$ is 1 or -1. This leads us to the following equality
$$\sigma_1 \cdot T^{2}(b) = \sigma_1 \cdot \prod_{k=1}^n{\sigma_{i_k+2}^{\epsilon_k}}$$ 
From the definition of  $B_{\infty}$, we know that  $\sigma_i$ and $\sigma_j$ commute with respect to the product operation if $|i -j|\geq 2$. This implies that for $i_k \in \mathbb{N}$, 
$$ \sigma_1 \cdot {\sigma_{{i_k}+2}^{\epsilon_k}} =  {\sigma_{i_k+2}^{\epsilon_k}} \cdot \sigma_1$$ 
which further implies that
$$ \sigma_1 \cdot \prod_{k=1}^n{\sigma_{{i_k}+2}^{\epsilon_k}} =  \prod_{k=1}^n{\sigma_{i_k+2}^{\epsilon_k}} \cdot \sigma_1$$ .
Since $b$ was an arbitrarily chosen element of $B_{\infty}$, it follows that
$$\sigma \cdot T^{2}(b) =  T^{2}(b) \cdot \sigma \quad \forall b\in B_{\infty} $$ .
\end{proof}

\begin{definition}[Canonical Model] 
For any signature $L$ and a set of sentences $\mathbb{T}$ in the language $L$, a structure $A$ is said to be a \emph{canonical model} if
\begin{itemize}
\item $A \models \mathbb{T}$
\item Every element of $A$ is of the form $t^A$, where $t$ is a closed term of $L$.
\item If $B$ is an $L$-structure and $B\models \mathbb{T}$, there is a unique homomorphism of structures $f:A\rightarrow B$.
\end{itemize}
\end{definition}

Now we prove that $(B_\infty, T, \cdot ,\equiv, 1, \sigma_1, \sigma_1^{-1})$ is a canonical model for the link axioms. In order to do so, consider a model of the link axioms $(S, T_1, *, \equiv', 1', \sigma_1', \bs_1')$ and a map $f :B_\infty \rightarrow S$ such that
 \begin{align*}
 f(1) &= 1'\\
 f(\sigma_1) &= \sigma_1'\\
 f(\sigma_i) &= T_1^{i-1}(\sigma_1')\\
 f(b_1 \cdot b_2) &= f(b_1)*f(b_2)
 \end{align*}
The last condition suffices to extend the function to $B_\infty$ using the image on the generating set. However it remains to be proved that it is well defined.
\begin{lemma} The map $f:B_{\infty} \rightarrow S$ is well defined.
\end{lemma}
\begin{proof}
In order to prove that the map is well defined, it suffices to prove that 
\begin{enumerate}
\item $T_1^{i}(\sigma_1')*T_1^{j}(\sigma_1') = T_1^{j}(\sigma_1')*T_1^{i}(\sigma_1')$,  For $i,j$ such that $j\geq i+2$ 
\item $T_1^{i}(\sigma_1')*T_1^{i+1}(\sigma_1')* T_1^{i}(\sigma_1')$ = $T_1^{i+1}(\sigma_1')*T_1^{i}(\sigma_1')*T_1^{i+1}(\sigma_1') $ 
\end{enumerate}
Since $(S, T_1, *, \equiv', 1' \sigma_1', \bs_1')$ is a model of the link axioms,
\[ \sigma_1' \cdot T^2(b) = T^2(b) \cdot \sigma_1', \quad \forall b\in S \]
Consider two arbitrarily chosen natural numbers $i$ and $j$ such that $j-i-2\geq 0$. If we substitute $T_1^{j-i-2}(\sigma_1')$ for $b$ we get,
\[ \sigma_1' \cdot T_1^2(T_1^{j-i-2}(\sigma_1')) = T_1^2(T_1^{j-i-2}(\sigma_1')) \cdot \sigma_1'\]
Which further leads to the equality,
\[T_1^i(\sigma_1' \cdot (T_1^{j-i}(\sigma_1')) = T_1^i((T_1^{j-i}(\sigma_1')) \cdot \sigma_1')\]
From the homomorphism action we get,
\[T_1^i(\sigma_1') \cdot T_1^i(T_1^{j-i}(\sigma_1')) = T_1^i(T_1^{j-i}(\sigma_1')) \cdot T_1^i(\sigma_1')\]
\[T_1^i(\sigma_1') \cdot T_1^j(\sigma_1') = T_1^j(\sigma_1') \cdot T_1^i(\sigma_1')\] 
Thus condition (1) holds. From the first braid axiom it follows that,
\[ \sigma_1' \cdot T_1(\sigma_1') \cdot \sigma_1' = T_1(\sigma_1') \cdot \sigma_1' \cdot T_1(\sigma_1')\]
Given $i\in\mathbb{N}$, $T_1^i$ is a homomorphism. By applying $T_1^i$ to the both the sides of the above equation, 
\[T_1^{i}(\sigma_1')*T_1^{i+1}(\sigma_1')* T_1^{i}(\sigma_1') = T_1^{i+1}(\sigma_1')*T_1^{i}(\sigma_1')*T_1^{i+1}(\sigma_1')\]
Thus $f$ is well defined. 
\end{proof}

\begin{theorem} $(B_\infty, T, \cdot ,\equiv, 1, \sigma_1, \sigma_1^{-1})$ is a canonical model for link axioms. 
\end{theorem}
\begin{proof}
From Theorem ~\ref{binfinitylink}, we know that $(B_\infty, T, \cdot ,\equiv, \sigma_1, \sigma_1^{-1})$ is a model of the link axioms and thus the first axiom in the definition of a canonical model holds true. Since every element in $B_\infty$ is of the form $\prod_{k=1}^n T^{i_k}(\sigma_1^{\epsilon_k})$, where $i_k \in \mathbb{N}\cup\{0\}$ and  $\epsilon_k$ is 1 or -1, it follows that every element is a closed term of $L$. Let $\prod_{k=1}^n{\sigma_{i_k}^{\epsilon_k}}$ be an arbitrarily chosen element of $B_\infty$. From the definition of $T$ it follows that
\begin{align*}
f\circ T(\prod_{k=1}^n{\sigma_{i_k}^{\epsilon_k}}) &= f(\prod_{k=1}^n{\sigma_{i_k+1}^{\epsilon_k}})\\
\end{align*}
Since $f$ is a group homomorphism,
\begin{align*}
f(\prod_{k=1}^n{\sigma_{i_k+1}^{\epsilon_k}})&=  \prod_{k=1}^n f({\sigma_{i_k+1}^{\epsilon_k}})\\
  &= \prod_{k=1}^n T^{{i_k}+1}_1(\sigma_{1}'^{\epsilon_k})\\
&= \prod_{k=1}^n T_1\circ T_1^{i_k}({\sigma_{1}'^{\epsilon_k}}) \\
 &= T_1\circ f(\prod_{k=1}^n{\sigma_{i_k}^{\epsilon_k}})  \\
\end{align*}
Which implies that
\begin{align*}
f\circ T &= T_1\circ f
\end{align*}
In order to prove that $f$ is a homomorphism of models, it now suffices to prove that if two elements are related by any of the generating relations of the stable equivalence, then their images under $f$ are related to each other. Since $(S, T_1, * , \equiv', 1, \sigma'_1, \bs_1')$ satisfies the Markov move axioms and $f$ preserves multiplication and identity, 
$$x, y, z, \in B_{\infty}, ((y\cdot z = 1) \implies (f(y \cdot x \cdot z) \equiv' f(x)))$$
From the equality it follows that $f\circ T = T_1\circ f$, 
$$ \sigma'_1*T_1(f(x)) = f(\sigma_1 \cdot T(f(x)))$$
and
$$\bs'_1*T_1(f(x)) = f(\sigma_1^{-1} \cdot T(f(x)))$$
hence it follows that
$$ f(x) \equiv' f(\sigma_1 \cdot T(f(x)))$$ 
 
$$ f(x) \equiv' f(\sigma_1^{-1} \cdot T(f(x)))$$

The uniqueness of $f$ remains to be proved. Any homomorphism $g$ between the given models maps $\sigma_1$ to $\sigma_1'$ and has to satisfy the condition $g\circ T$ = $T_1\circ g$. By applying induction, it follows that $g\circ T^i(\sigma_1) = T^i(g(\sigma_1)) = T^i(f(\sigma_1))$ for every $i\in \mathbb{N}$. Thus for any given homomorphism $g$ and any given $i\in \mathbb{N}$, image of $g(\sigma_i)$ is $T_1^i(f(\sigma_1))$. Since the elements of the set $\{\sigma_i\}_{i\in \mathbb{N}}$ generate $B_{\infty}$, it follows that $g = f$. 
\end{proof}

\section*{Appendix: Topological background} \label{append}

\subsection{Knots and Braids}

In order to understand how the above axioms constitute a model of the knots, we introduce the following conceptual apparatus through which knots can be reduced to braid groups.

\begin{definition}[Braid Diagrams]
 A braid diagram is a set $D\subset \mathbb{R}\times I$ split into topological intervals called the strands of D such that
\begin{enumerate}
 \item The projection $\mathbb{R}\times I\rightarrow I$ maps each strand homeomorphically onto I.
\item Every point of $\mathbb{N}\times \{0,1\}$ is an end point of a unique strand.
\item\label{trans} Every point of $\mathbb{R}\times I$ belongs to at most two strands. At each intersection point of two strands, these strands meet transversally. At every intersection, one of the intersecting strands is labelled 'undergoing' and the other strand is labelled 'overgoing'. 
\end{enumerate}
\end{definition}
\begin{remark}
Transversality  in condition ~\ref{trans} means that in a neighbourhood of a crossing, up to homeomorphism, D is like the set $\{(x,y)| xy=0\}$ . Compactness of strands and condition(iii) ensures that the number of double points are finite.
\end{remark}

Two braid diagrams $D$ and $D'$ are said to be \textit{isotopic} if there is a continuous map $F:D\times I \rightarrow \mathbb{R}\times I$ such that for each $s\in I$, $D_s = F(D\times \{s\})\subset \mathbb{R}\times I$ is a braid diagram preserving $D_0 = D$ and $D_1 = D'$. It is understood that F maps the crossings of D to crossings of $D_s$ while preserving the information about when a strand goes under or over the other strand.An example of braid diagram isotopy is given in the figure below.

\begin{center}
\includegraphics[scale=0.40]{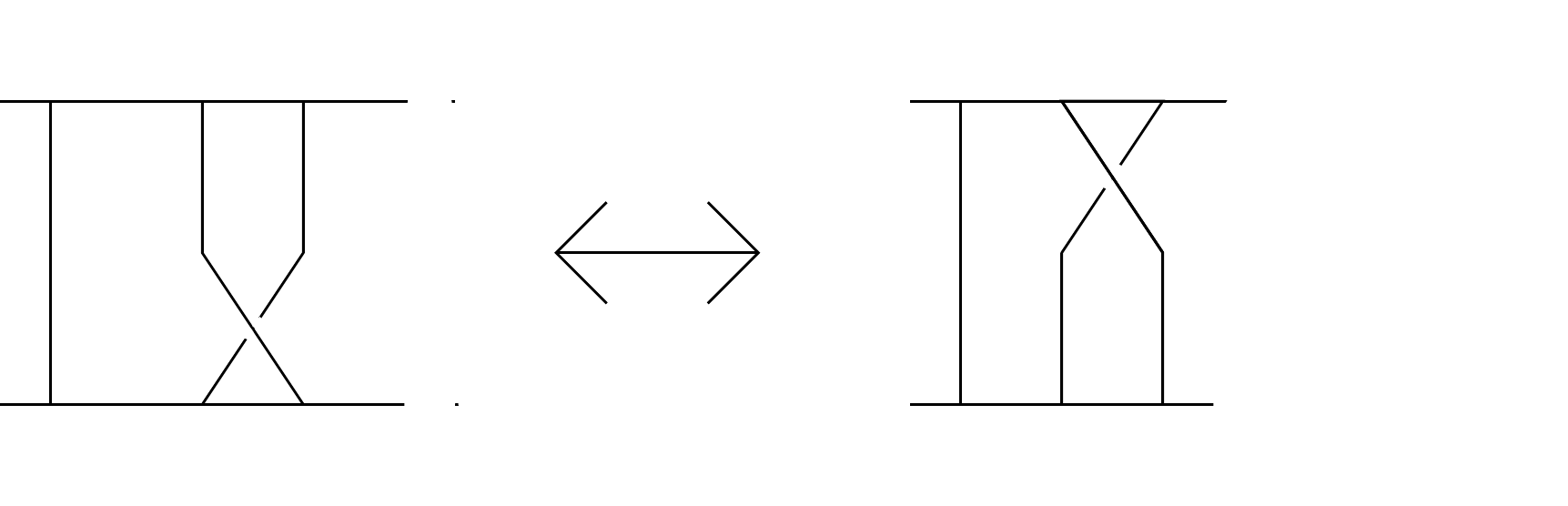}
\end{center}

The term braid diagram shall be here on used to denote the isotopy class of a braid diagram. A product $*$ on the isotopy classes of braid diagrams is defined by associating the diagram $D_1*D_2$ to the diagram obtained by placing $D_1$ on top of $D_2$ and 'squeezing' the resultant diagram into $\mathbb{R}\times I $, i.e., (x, t)$\rightarrow$(x, t/2). It is easy to see that the product is well defined and the braid diagram $e = \cup_{k=1}^n \{(k, t)| t\in [0,1]\}$, is the identity. The product is associative because any element of the form $(a*b)*c$ can be isotoped to $a*(b*c)$ through a continuous map defined as follows. 

$$F(x,t,\theta)=\begin{cases}
(1-\theta)(x, t) + \theta(x, 2t), & \quad t\in [0,1/4],\\
(1-\theta)(x, t) + \theta(x, t + 1/4), & \quad t \in [1/4,1/2],\\
(1-\theta)(x, t) + \theta(x, (t+1)/2), & \quad t\in [1/2, 1].\\
\end{cases}$$

To define a map from Braid Group $B_n$ to the set of n-braid diagrams $\mathtt{B_n}$ upto isotopy, we define an equivalence relation on $\mathtt{B_n}$ generated by a set of moves called the \textit{reidemeister moves}.

\begin{definition}
[Reidemeister Moves]
The transformations of the braid diagrams $\Omega_2, \Omega_3$ as shown in the figure below and their inverses $\Omega_2^{-1}, \Omega_3^{-1}$ are called \textit{Reidemeister moves}. 

\begin{center}
\includegraphics[scale=0.30]{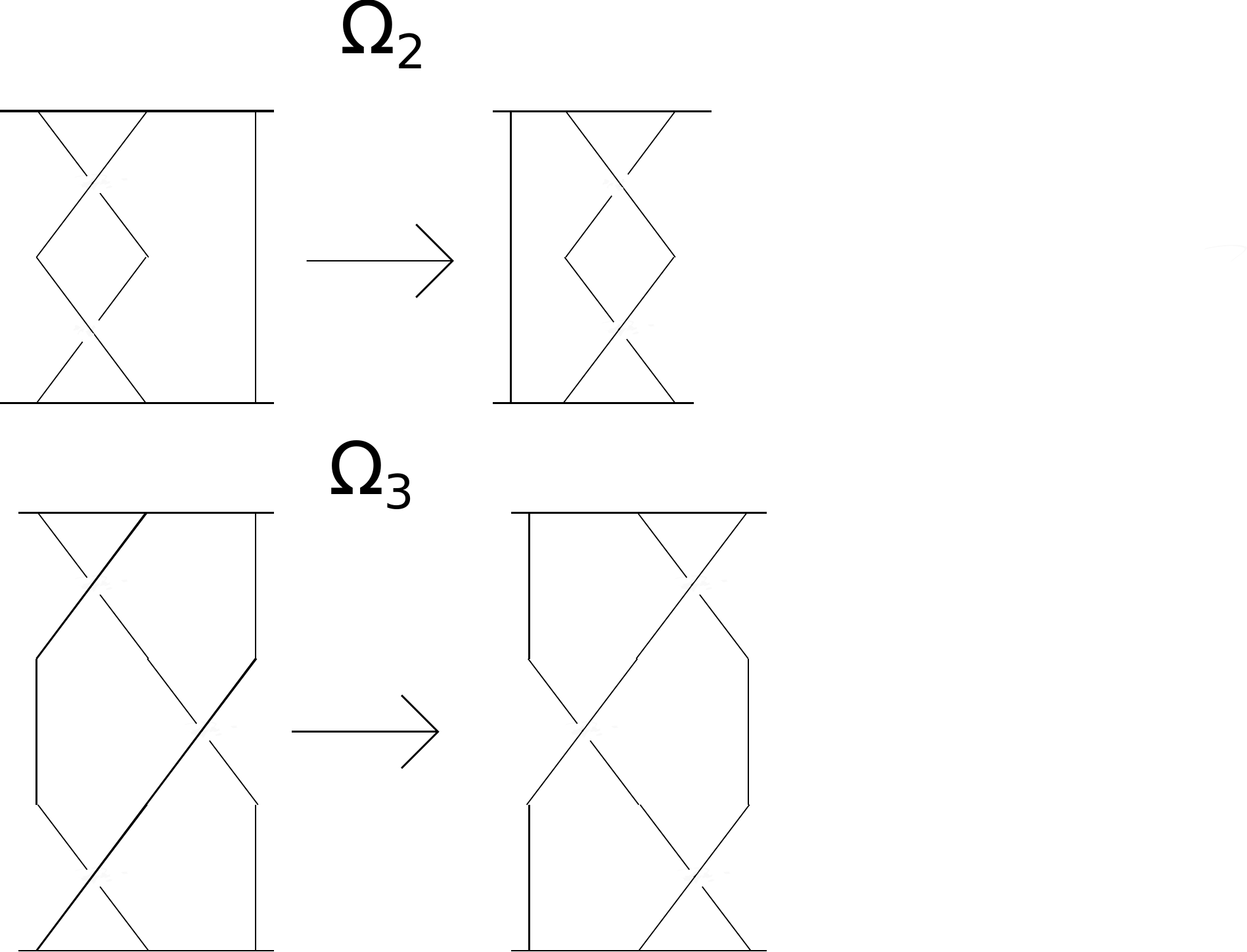}
\end{center}

Two distinct braid diagrams are said to be \textit{R-equivalent} if they are related to each other by a sequence of Reidemeister moves.
\end{definition}

Observe that if there exist braid diagrams a, b, c and d such that a is R- equivalent to c and b is R-equivalent to d, $a*b$ is R-equivalent to $c*d$. It follows that the product $*$  on $\mathtt{B_n}$, the isotopy classes of braid diagrams upto R-equivalence, is well defined and $\mathtt{B}_n$ is a monoid. Consider the braid diagrams $\overline{\sigma}_i^+$ and $\overline{\sigma}_i^-$ as in the figure below.

\begin{center}
\includegraphics[scale=0.30]{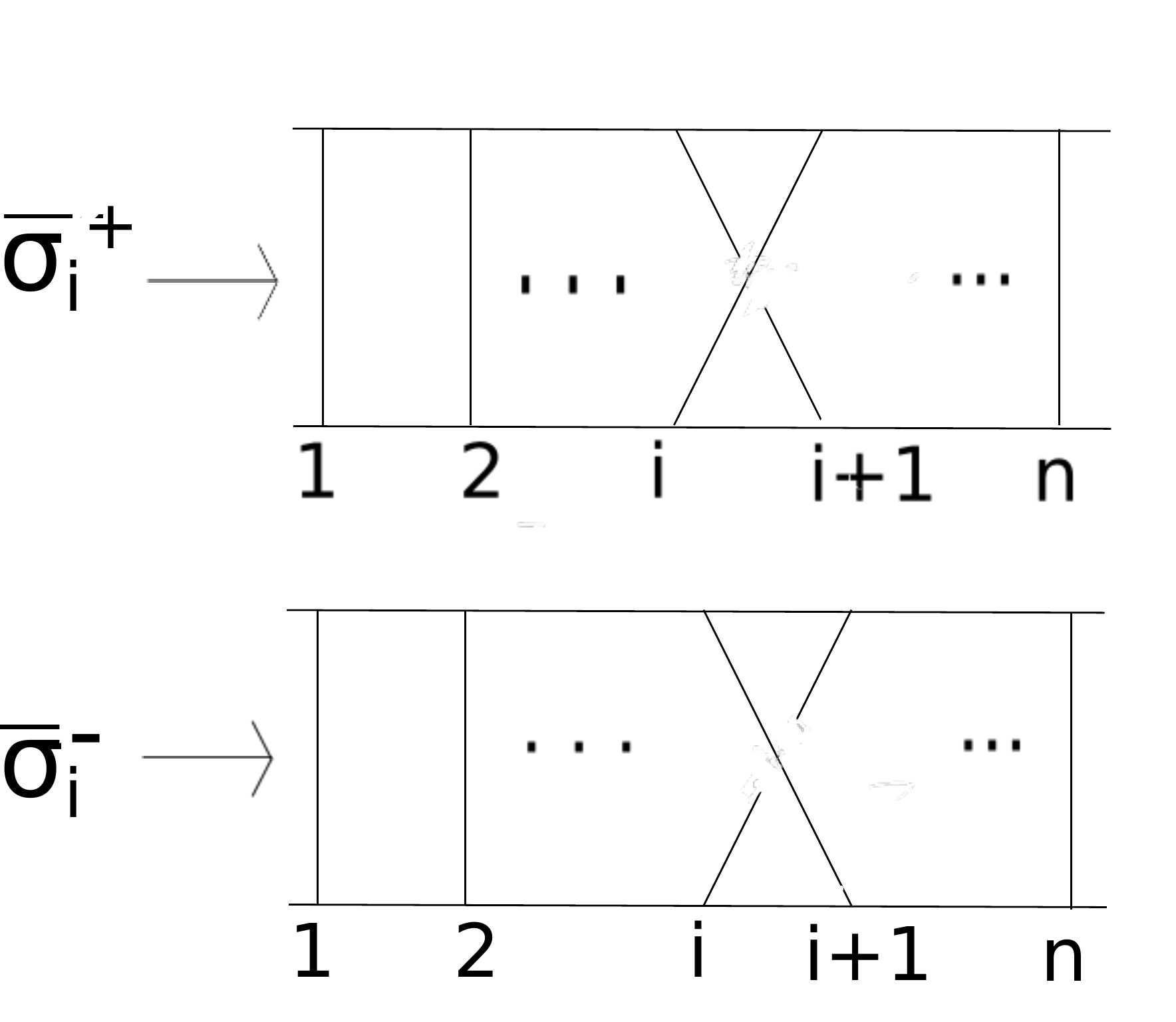}
\end{center}

The following result implies that the elementary braid diagrams  $\overline{\sigma}_i^+$ and $\overline{\sigma}_i^-$ generate $\mathtt{B}_n$ and $(\mathtt{B}_n, *)$  is a group.

\begin{lemma} R-equivalent classes of elementary braid diagrams $\overline{\sigma}_i^+$ and $\overline{\sigma}_i^-$ generate $\mathtt{B}_n$ as a monoid. Further each element $\beta\in \mathtt{B}_n$ has a two sided inverse $\beta^{-1}$ in $\mathtt{B}_n$.
\end{lemma}

The maps  $\phi_{1_{n_\epsilon}}: B_n \to \mathtt{B}_n$ such that $\phi_{1_{n_\epsilon}}(\sigma_i) = \overline{\sigma}^{\epsilon}_i$,  are well defined because the first axiom of braid groups corresponds to isotopy of braid diagrams and the second axiom corresponds to the second Reidemeister move. The following result says that it is an isomorphism of groups. 

\begin{theorem}
$\phi_{1_{n_\epsilon}}$ is an isomorphism of groups. 
\end{theorem}

To construct Links from braid diagrams, we define the following set of objects in $\mathbb{R}^2\times I$ called Geometric Braids.

\begin{definition}
[Geometric Braids]
A geometric braid on $n\geq 2$ strings is a set $b \subset \mathbb{R}^2\times I$ formed by the $n$ disjoint topological intervals called the \textit{strings} of $b$, such that
\begin{enumerate}
 \item $\pi_3: \mathbb{R}^2\times I\rightarrow I$ ,i.e., the projection onto I, maps each string homeomorphically onto I.
\item $b \cap (\mathbb{R}^2\times \{0\}) = \{(k,0,0)\}_{k=1}^n$
\item $b \cap (\mathbb{R}^2\times \{1\}) = \{(k,0,1)\}_{k=1}^n$
\end{enumerate}

Two geometric braids $b$ and $b'$ are said to be \textit{isotopic} if there is a continuous map, $F:b\times I \rightarrow R^{2}\times I$ such that
\begin{enumerate}
 \item $F_s: b\rightarrow \mathbb{R}^2\times I$ sending $x\in b$ to F(x, s) is an embedding, whose image is a geometric braid on n strings.
\item $F_0 = id_0:b\rightarrow b$
\item $F_1(b) = b'$
\end{enumerate}
\end{definition}

\begin{theorem}
 F extends to an isotopy of $\mathbb{R}^2\times I$ which is identity on boundaries.
\end{theorem}

The map $i:\mathbb{R}\times I -\mathbb{R}^2\times I$ given by $i(x,t)=(x,0,t)$ embeds braid diagrams in $R^2\times I$. Pushing the undergoing strand in a small neighbourhood of every crossing into $\mathbb{R}\times (0, \infty)\times I$ by appropriately increasing the second co-ordinate of the strand while keeping the first and third constant, one obtains a geometric n-braid. Observe that the geometric braids constructed from isotopic braid diagrams are isotopic.
\begin{center}
\includegraphics[scale=0.35]{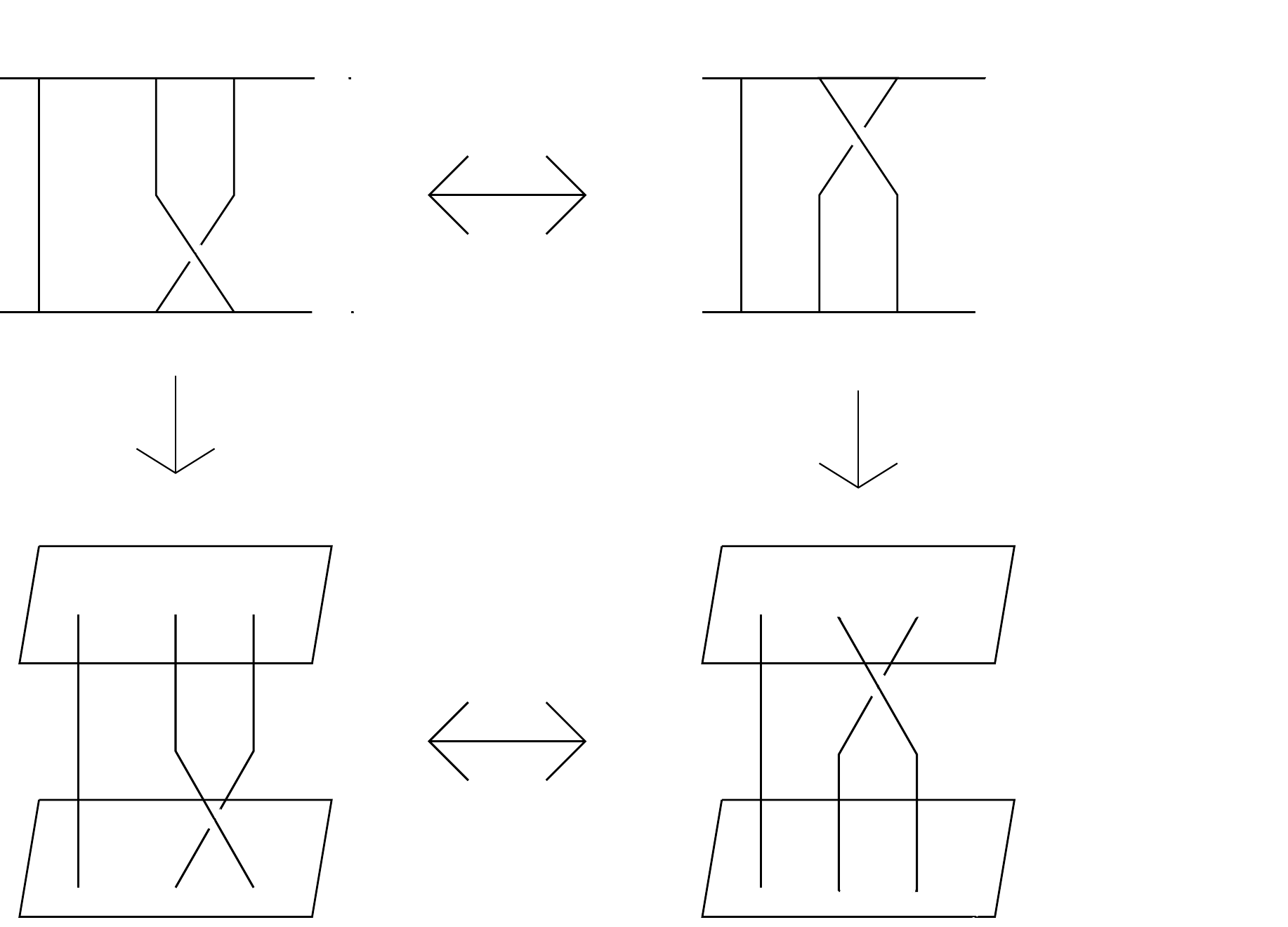}
\end{center}
The above construction induces a map $\phi_{2_n}$ from the isotopy classes of n-braid diagrams to $\mathbb{B}_n$, the set of geometric n-braids upto isotopy. 
Projecting  a geometric braid onto its first and third co-ordinates and marking the intersecting strand with greater value of the second co-ordinate in a neighbourhood of intersection as undercrossing and the other as overcrossing in the neighbourhood, we obtain a braid diagram. The image of this braid diagram under $\phi_{2_n}$ is isotopic to the original geometric braid. Thus the map $\phi_{2_n}$ is surjective. 

The following theorem makes explicit the relationship between braid diagrams and geometric braids.
\begin{theorem} $\phi_{2_n}(b_1) =\phi_{2_n}(b_2)$ if and only if $b_1$ is R-equivalent to $b_2$.
\end{theorem}

\begin{corollary}$\phi_{2_n}\circ \phi_{1_{n_\epsilon}}: B_n \rightarrow \mathbb{B}_n$ is bijective.
\end{corollary}

If $f:(0,\infty)\to (0,1)$ is an order preserving homeomorphism, then the map $\phi_{3_n}:R^2\times I\to D\times  S^1$ where 
$$\phi_{3_n}(r(cos(\theta),sin(\theta)), t) = (f(r)(cos(\theta),sin(\theta)), e^{it})$$
induces a well defined map from geometric $n$-braids to isotopic classes of Links. 

The map $\phi:\bds \to L$ obtained by composition of the above maps, $\phi(b , n) = \phi_{3_n} \circ \phi_{2_n}\circ\phi_{1_{n_+}}(b)$, assigns an isotopy class of links to each braid which could be labelled as 'closing the braid'. Markov's Theorem and Alexander's Theorem can thus be reformulated in terms of $\phi$.

\begin{theorem}[Alexander] $\phi$ is surjective.
 \end{theorem}
 
\begin{theorem}[Markov] $\phi(b_1,m_1) = \phi(b_2,m_2)$ if and only if $(b_1,m_1)\sim(b_2,m_2)$.
 \end{theorem} 

\nocite{kassel2008braid}
\nocite{hodges1993model}
\bibliographystyle{plain}
\bibliography{KnotsPredicate}

\end{document}